\documentclass{amsart}
\usepackage{latexsym,amssymb,amsmath,amsthm,amscd,graphicx,esint}

\makeatletter
\@namedef{subjclassname@2010}{%
  \textup{2010} Mathematics Subject Classification}
\makeatother

\numberwithin{equation}{section}
\newtheorem{theorem}{Theorem}[section]
\newtheorem{lemma}[theorem]{Lemma}
\newtheorem{corollary}[theorem]{Corollary}
\newtheorem{proposition}[theorem]{Proposition}

\theoremstyle{definition}

\newtheorem{remark}[theorem]{Remark}
\newtheorem{definition}[theorem]{Definition}

\theoremstyle{remark}

\newcommand{\cB}{{\mathcal B}}

\newcommand{\cD}{{\mathcal D}}

\newcommand{\cP}{{\mathcal P}}
\newcommand{\cQ}{{\mathcal Q}}

\newcommand{\N}{{\mathbb N}}

\newcommand{\R}{{\mathbb R}}

\newcommand{\Z}{{\mathbb Z}}

\def\eps{\varepsilon}

\def\lm{\lambda}

\def\0{\emptyset}
\def\6{\partial}
\def\8{\infty}

\def\l{\left}
\def\r{\right}

\DeclareMathOperator*{\esssup}{ess\,sup}
\DeclareMathOperator*{\essinf}{ess\,inf}

\begin{document}

\title
[The Kakeya maximal operator on the variable Lebesgue spaces]
{The Kakeya maximal operator on the variable Lebesgue spaces}

\author[H.~Saito]{Hiroki Saito}
\address{Department of Mathematics and Information Sciences, Tokyo Metropolitan University, 1-1 Minami Ohsawa, Hachioji, Tokyo 192-0397, Japan}
\email{j1107703@gmail.com}
\author[H.~Tanaka]{Hitoshi Tanaka}
\address{Graduate School of Mathematical Sciences, The University of Tokyo, Tokyo, 153-8914, Japan}
\email{htanaka@ms.u-tokyo.ac.jp}

\thanks{
The second author is supported by 
the FMSP program at Graduate School of Mathematical Sciences, the University of Tokyo, 
and Grant-in-Aid for Scientific Research (C) (No.~23540187), 
the Japan Society for the Promotion of Science. 
}


\date{}

\begin{abstract}
We shall verify 
the Kakeya (Nikodym) maximal operator 
$K_{N}$, $N\gg 1$, 
is bounded on the variable Lebesgue space 
$L^{p(\cdot)}(\R^2)$ 
when the exponent function $p(\cdot)$ is 
$N$-modified locally log-H\"{o}lder continuous and 
log-H\"{o}lder continuous at infinity. 
\end{abstract}

\maketitle

\section{Introduction}\label{sec1}
The purpose of this paper is to investigate 
the boundedness of the Kakeya (Nikodym) maximal operator on the variable Lebesgue spaces. 
Given a measurable function 
$p(\cdot):\,\R^n\to[1,\8)$, 
we define the variable Lebesgue space 
$L^{p(\cdot)}(\R^n)$ 
to be the set of measurable functions such that 
for some $\lm>0$,
$$
\rho_{p(\cdot)}(f/\lm)
=
\int_{\R^n}\l(\frac{|f(x)|}{\lm}\r)^{p(x)}\,dx
<\8.
$$
$L^{p(\cdot)}(\R^n)$ is a Banach space when equipped with the norm 
$$
\|f\|_{p(\cdot)}
=
\|f\|_{L^{p(\cdot)}(\R^n)}
=
\inf\{\lm>0:\,\rho_{p(\cdot)}(f/\lm)\le 1\}.
$$
The variable Lebesgue space 
$L^{p(\cdot)}(\R^n)$ 
generalizes the classical Lebesgue space 
$L^p(\R^n)$: 
if $p(\cdot)\equiv p_0$, then 
$L^{p(\cdot)}(\R^n)=L^{p_0}(\R^n)$.
Variable Lebesgue spaces have been studied in the past twenty years 
(see \cite{CaCrFi,CrDiFi,CrFiMaPe,Die1,Die2,Die3,DiHaNe,KR,Le,Ne}).
For a locally integrable function $f$ on $\R^n$ 
the Hardy-Littlewood maximal operator $M$ is defined by 
$$
Mf(x)
=
\sup_{x\in Q\in\cQ}
\fint_{Q}|f(y)|\,dy,
$$
where we have used 
$\cQ$ to denote the family of all cubes in $\R^n$ 
with sides parallel to the coordinate axes 
and $\fint_{Q}f(x)\,dx$ 
to denote the usual integral average of $f$ over $Q$.
Let $\cP(\R^n)$ be the class of all functions $p(\cdot)$ 
for which the Hardy-Littlewood maximal operator $M$ 
is bounded on $L^{p(\cdot)}(\R^n)$.
By the classical Hardy-Littlewood maximal theorem, 
any constant function $p(\cdot)\equiv p_0$ 
with $1<p_0<\8$ belongs to $\cP(\R^n)$. 
In \cite{Die2}, 
L.~Diening showed that 
$p(\cdot)\in\cP(\R^n)$ 
if and only if 
there exists a positive constant $c$ 
such that 
for any family of pairwise disjoint cubes $\pi$ and 
any $f\in L^{p(\cdot)}(\R^n)$,
$$
\l\|\sum_{Q\in\pi}\fint_{Q}|f(y)|\,dy\,\chi_{Q}\r\|_{L^{p(\cdot)}(\R^n)}
\le c
\|f\|_{L^{p(\cdot)}(\R^n)},
$$
where $\chi_{E}$ stands for the characteristic function of 
a measurable set $E\subset\R^n$.
This result implies, for example, that 
$p(\cdot)\in\cP(\R^n)$
if and only if 
$p'(\cdot)\in\cP(\R^n)$, where 
$p'(x)=\frac{p(x)}{p(x)-1}$.
However, 
since this result is very general, 
some simple sufficient conditions 
for which $p(\cdot)\in\cP(\R^n)$ 
has been studied by many authors 
(see \cite{Die1,CrFiNe,Le,Ne}).
In \cite{CrFiNe}, 
D.~Cruz-Uribe, A.~Fiorenza and C.~J.~Neugebauer 
give a new and simpler proof of 
the boundedness of the Hardy-Littlewood maximal operator $M$ 
on variable Lebesgue space 
$L^{p(\cdot)}(\R^n)$.

\begin{definition}\label{def1.1}
\begin{description}
\item[{\rm(a)}] 
We say that $p(\cdot)$ is 
locally log-H\"{o}lder continuous 
if there exists a positive constant $c_0$ 
such that 
$$
|p(x)-p(y)|\log\l(\frac{1}{|x-y|}\r)
\le c_0,\qquad
x,y\in\R^n,\,|x-y|<1.
$$
\item[{\rm(b)}] 
We say that $p(\cdot)$ is 
log-H\"{o}lder continuous at infinity 
if there exist constants $c_{\8}$ and $p(\8)$ 
such that 
$$
|p(x)-p(\8)|\log(e+|x|)\le c_{\8},
\qquad x\in\R^n.
$$
\item[{\rm(c)}] 
Given a measurable set $E\subset\R^n$, 
let 
$$
p_{-}(E)=\essinf_{x\in E}p(x)
\text{ and }
p_{+}(E)=\esssup_{x\in E}p(x).
$$
If $E=\R^n$, then 
we simply write $p_{-}$ and $p_{+}$. 
\end{description}
\end{definition}

\begin{proposition}[{\rm\cite[Theorem 1.2]{CrFiNe}}]\label{prp1.2}
Let $1<p_{-}\le p_{+}<\8$. 
Suppose that $p(\cdot)$ is 
locally log-H\"{o}lder continuous and 
log-H\"{o}lder continuous at infinity. 
Then there exists a positive constant $C$ 
independent of $f$ such that 
$$
\|Mf\|_{L^{p(\cdot)}(\R^n)}
\le C
\|f\|_{L^{p(\cdot)}(\R^n)}.
$$
\end{proposition}

For a locally integrable function $f$ on $\R^2$ 
the Kakeya (Nikodym) maximal operator 
$K_{N}$, $N\gg 1$, 
is defined by 
$$
K_{N}f(x)
=
\sup_{x\in R\in\cB_{N}}
\fint_{R}|f(y)|\,dy,
$$
where $\cB_{N}$ denotes 
the set of all rectangles in $\R^2$ 
with eccentricity $N$ 
(the ratio of the length of long-sides and short-sides is equal to $N$).
In this paper, 
we investigate the boundedness property of 
the Kakeya maximal operator $K_{N}$ 
on the variable Lebesgue space 
$L^{p(\cdot)}(\R^2)$. 
It is well known that 
(see \cite{Co,Gr,Str})
\begin{equation}\label{1.1}
\|K_{N}f\|_{L^p(\R^2)}
\le C_p
(\log N)^{2/p}
\|f\|_{L^p(\R^2)}
\text{ for }
2\le p\le\8.
\end{equation}
One might naturally expect that 
$$
\|K_N\|_{L^{p(\cdot)}(\R^2)\to L^{p(\cdot)}(\R^2)}
\le C
(\log N)^{2/p_{-}}
\text{ when }
2\le p_{-}\le p_{+}<\8,
$$
where 
$\|T\|_{L^{p(\cdot)}(\R^2)\to L^{p(\cdot)}(\R^2)}$ 
denotes the operator norm 
$T:\,L^{p(\cdot)}(\R^2)\to L^{p(\cdot)}(\R^2)$. 
However, 
we have the following theorem. 

\begin{theorem}\label{thm1.3}
Let $N\gg 1$ and 
$1<p_{-}<p_{+}<\8$. 
Suppose that $K_{N}$ is bounded 
from $L^{p(\cdot)}(\R^2)$ 
to $L^{p(\cdot)}(\R^2)$ and that 
$p(\cdot)$ is continuous. 
Then there exist a positive constant $C$, 
independent of $N$, and 
a small constant $\eps>0$ such that 
$$
\|K_N\|_{L^{p(\cdot)}(\R^2)\to L^{p(\cdot)}(\R^2)}
\ge C
N^{\eps}.
$$
\end{theorem}
Thus, 
in the framework of the variable Lebesgue spaces,
we are interested in a small positive  constant $c$ 
such that $N^c$ bounds from above 
$\|K_N\|_{L^{p(\cdot)}(\R^2)\to L^{p(\cdot)}(\R^2)}$.

The main result of this paper is the following (Theorem \ref{thm1.5}). 
The technique of the proof of our theorem 
is due to \cite{CrDiFi}, which is used 
the machinery of Calder\'{o}n-Zygmund cubes.
We apply this technique to the rectangles in $\cB_{N}$.
For the precise estimate we need the following notion. 

\begin{definition}\label{def1.4}
Let $N\gg 1$. 
We say that $p(\cdot)$ is 
$N$-modified locally log-H\"{o}lder continuous 
if there exists a positive constant $c_{N}$ 
such that 
$$
\l|\frac{1}{p(x)}-\frac{1}{p(y)}\r|
\log\l(\frac{N}{|x-y|^2}\r)
\le c_{N}\log N,
\qquad
x,y\in\R^2,\,|x-y|<\sqrt{N}.
$$
\end{definition}

\begin{theorem}\label{thm1.5}
Let $N\gg 1$ and 
$2\le p_{-}\le p_{+}<\8$. 
Suppose that $p(\cdot)$ is 
$N$-modified locally log-H\"{o}lder continuous and 
log-H\"{o}lder continuous at infinity. 
Then $K_{N}$ is bounded 
from $L^{p(\cdot)}(\R^2)$ 
to $L^{p(\cdot)}(\R^2)$ and 
$$
\|K_N\|_{L^{p(\cdot)}(\R^2)\to L^{p(\cdot)}(\R^2)}
\le C
N^{p_{-}c_{N}}(\log N)^{2/p_{-}},
$$
where the constant $C$ is independent of $N$. 
\end{theorem}

\begin{corollary}\label{cor1.6}
Let $N\gg 1$ and 
$2\le p_{-}\le p_{+}<\8$. 
Suppose that $p(\cdot)$ is 
locally log-H\"{o}lder continuous and 
log-H\"{o}lder continuous at infinity. 
Then $K_{N}$ is bounded 
from $L^{p(\cdot)}(\R^2)$ 
to $L^{p(\cdot)}(\R^2)$ and 
$$
\|K_N\|_{L^{p(\cdot)}(\R^2)\to L^{p(\cdot)}(\R^2)}
\le C
N^{p_{-}C_{N}}(\log N)^{2/p_{-}},
$$
where the constant $C$ is independent of $N$ and 
$$
C_{N}
=
\sup_{\substack{x,y\in\R^2 \\ |x-y|<\sqrt{N}}}
\l|\frac{1}{p(x)}-\frac{1}{p(y)}\r|.
$$
\end{corollary}

\begin{remark}\label{rem1.7}
We remark that 
$$
p_{-}C_{N}
\le
p_{-}\l(\frac{1}{p_{-}}-\frac{1}{p_{+}}\r)
=
1-\frac{p_{-}}{p_{+}}.
$$
\end{remark}

The letter $C$ will be used for constants 
that may change from one occurrence to another. 
Constants with subscripts, such as $C_1$, $C_2$, do not change 
in different occurrences. 

\section{Proof of Theorem \ref{thm1.3}}\label{sec2}
The following argument is due to T.~Kopaliani \cite{Ko} (see also \cite{KM}).
Recall that the conjugate function $p'(x)$ is defined by 
$\frac{1}{p'(x)}+\frac{1}{p(x)}=1$.
The following generalized H\"{o}lder inequality and 
a duality relation can be found in \cite{KR}:
$$
\int_{\R^2}|f(x)g(x)|\,dx
\le 2
\|f\|_{p(\cdot)}\|g\|_{p'(\cdot)},
$$
$$
\|f\|_{p(\cdot)}
\le
\sup_{\|g\|_{p'(\cdot)}\le 1}
\int_{\R^2}|f(x)g(x)|\,dx.
$$
Suppose that $K_{N}$ is bounded 
from $L^{p(\cdot)}(\R^2)$ 
to $L^{p(\cdot)}(\R^2)$. Then 
for every rectangle $R\in\cB_{N}$ 
we have
$$
\|K_{N}\|_{L^{p(\cdot)}\to L^{p(\cdot)}}
\ge
\|K_{N}f\|_{p(\cdot)}
\ge
\l\|
\fint_{R}f(y)\,dy\,\chi_{R}
\r\|_{p(\cdot)}
=
\fint_{R}f(y)\,dy\,
\|\chi_R\|_{p(\cdot)}
$$
for all nonnegative $f$ with 
$\|f\|_{p(\cdot)}\le 1$.
Taking supremum all such $f$,
we have
\begin{equation}\label{2.1}
\|K_{N}\|_{L^{p(\cdot)}\to L^{p(\cdot)}}
\ge
\frac{1}{|R|}
\|\chi_{R}\|_{p'(\cdot)}
\|\chi_R\|_{p(\cdot)}
\end{equation}
for all $R\in\cB_{N}$,
where $|R|$ denotes the area of the rectangle $R$.

Suppose that $p(\cdot)$ is continuous and is not constant. 
Then we can find two closed squares 
$B_1$ and $B_2$ in $\R^2$
with $|B_1|,|B_2|<1$,
such that
\begin{equation}\label{2.2}
p_{+}(B_1)<p_{-}(B_2).
\end{equation}
Without loss of generality we may assume that 
$$
B_1=[0,s]\times[0,s]
\text{ and }
B_2=[0,s]\times[t-s,t]
\text{ for some }
t>s>0.
$$
We take $N$ with $t/N<s$ and let 
$R=[0,t/N]\times[0,t]$.
Then we have 
$R\in\cB_{N}$ and 
$$
|R\cap B_1|
=
|R\cap B_2|
=
\frac{st}{N}.
$$
Observe now that the following embeddings hold:
\begin{align*}
L^{p(\cdot)}(B_2)
&\hookrightarrow
L^{p_{-}(B_2)}(B_2),
\\
L^{p'(\cdot)}(B_1)
&\hookrightarrow
L^{(p_{+}(B_1))'}(B_1),
\end{align*}
where 
$\frac{1}{(p_{+}(B_1))'}+\frac{1}{p_{+}(B_1)}=1$.
It follows that 
\begin{align*}
\frac{1}{|R|}
\|\chi_{R}\|_{p(\cdot)}
\|\chi_{R}\|_{p'(\cdot)}
&\ge
\frac{1}{|R|}
\|\chi_{R\cap B_2}\|_{L^{p(\cdot)}(B_2)}
\|\chi_{R\cap B_1}\|_{L^{p'(\cdot)}(B_1)}
\\ &\ge
\frac{1}{|R|}
\|\chi_{R\cap B_2}\|_{L^{p_{-}(B_2)}(B_2)}
\|\chi_{R\cap B_1}\|_{L^{(p_{+}(B_1))'}(B_1)}
\\ &=
|R|^{-1}
\cdot
|R\cap B_2|^{\frac{1}{p_{-}(B_2)}}
\cdot
|R\cap B_1|^{1-\frac{1}{p_{+}(B_1)}}
\\ &=
t^{-2}
\cdot
(st)^{1+\frac{1}{p_{-}(B_2)}-\frac{1}{p_{+}(B_1)}}
\cdot
N^{\frac{1}{p_{+}(B_1)}-\frac{1}{p_{-}(B_2)}},
\end{align*}
where we have used $|B_1|,|B_2|<1$. 
Since by \eqref{2.2} 
$\frac{1}{p_{+}(B_1)}-\frac{1}{p_{-}(B_2)}>0$,
we conclude by \eqref{2.1} that 
$\|K_{N}\|_{L^{p(\cdot)}\to L^{p(\cdot)}}$
has a lower bound $N^{\eps}$ with 
$\eps>0$.

\section{Proof of Theorem \ref{thm1.5}}\label{sec3}
In what follows we shall prove Theorem \ref{thm1.5}. 
We need two lemmas. 

\begin{lemma}\label{lem3.1}
Let $N\gg 1$. 
Suppose that $p(\cdot)$ is 
$N$-modified locally log-H\"{o}lder continuous. 
Then, for any rectangle $R\in\cB_{N}$, 
$$
|R|^{\frac{1}{p_{+}(R)}-\frac{1}{p_{-}(R)}}
\le N^{c_{N}}.
$$
\end{lemma}

\begin{proof}
When $|R|\ge 1$, there is nothing to prove. 
Suppose that $|R|<1$. 
Since $p(\cdot)$ is continuous, 
there exist $x,y\in R$ such that 
$p(x)=p_{-}(R)$ and $p(y)=p_{+}(R)$.
It follows that 
\begin{align*}
\lefteqn{
|R|^{\frac{1}{p_{+}(R)}-\frac{1}{p_{-}(R)}}
=
|R|^{\frac{1}{p(y)}-\frac{1}{p(x)}}
\le
\l(\frac{|x-y|^2}{N}\r)^{\frac{1}{p(y)}-\frac{1}{p(x)}}
}\\ &=
\exp\l\{
\l(\frac{1}{p(y)}-\frac{1}{p(x)}\r)
\log\l(\frac{|x-y|^2}{N}\r)
\r\}
=
\exp\l\{
\l(\frac{1}{p(x)}-\frac{1}{p(y)}\r)
\log\l(\frac{N}{|x-y|^2}\r)
\r\}
\\ &\le
\exp\{\log(N^{c_{N}})\}
=
N^{c_{N}},
\end{align*}
where we have used $|x-y|<\sqrt{N}$ and 
the $N$-modified local log-H\"{o}lder continuity of $p(\cdot)$.
\end{proof}

\begin{lemma}[{\rm\cite[Lemma 2.4]{CrDiFi}}]\label{lem3.2}
Suppose that $p(\cdot)$ is 
log-H\"{o}lder continuous at infinity. 
Let $P(x)=(e+|x|)^{-M}$, $M\ge 2$. 
Then there exists a constant $c$ depending on 
$M$, $p(\8)$ and $c_{\8}$ such that 
given any set $E$ and any function $F$ 
such that $0\le F(y)\le 1$, $y\in E$, 
$$
\int_{E}F(y)^{p(y)}\,dy
\le
c\int_{E}F(y)^{p(\8)}\,dy
+
c\int_{E}P(y)^{p(\8)}\,dy,
$$
$$
\int_{E}F(y)^{p(\8)}\,dy
\le
c\int_{E}F(y)^{p(y)}\,dy
+
c\int_{E}P(y)^{p(\8)}\,dy.
$$
\end{lemma}

\noindent
{\it Proof of Theorem \ref{thm1.5}.} 
We may assume that $f$ is nonnegative. 
We first linearize the operator $K_{N}$.
For $k\in\N$, we denote by $\cD_k$ 
the family of all dyadic cubes 
$Q=2^{-k}(m+[0,1)^2)$, $m\in\Z^2$. 
For each $Q\in\cD_k$ we choose 
a rectangle $R(Q)\in\cB_{N}$, 
such that $R(Q)\supset Q$. 
We denote the operator $T_k$ as 
$$
T_kf(x)
=
\sum_{Q\in\cD_k}
\fint_{R(Q)}f(y)\,dy\,\chi_{Q}(x).
$$
By definition it is easy to see that 
$$
T_kf(x)\le K_{N}f(x)
$$
for any choice of rectangles $\{R(Q)\}$. 
On the other hand, 
there is a sequence of linearized operators $\{T_kf\}$ 
which converge pointwise to $K_{N}f$ 
as $k$ tends to infinity. 
Thus, by the Fatou theorem 
we need only prove Theorem \ref{thm1.5} 
with $K_{N}$ replaced by $T_k$ 
with a constant $C$ not depending on $k$. 

By homogeneity we may assume that 
$\|f\|_{p(\cdot)}=1$. Then 
$$
\rho_{p(\cdot)}(f)
=
\int_{\R^2}f(x)^{p(x)}\,dx\le 1.
$$
Decompose $f$ as $f_1+f_2$, where 
$f_1=f\chi_{\{x:\,f(x)>1\}}$ 
and 
$f_2=f\chi_{\{x:\,f(x)\le 1\}}$.
Then 
$$
\rho_{p(\cdot)}(f_i)\le 1
\text{ for }
i=1,2.
$$
Let 
\begin{align*}
C_1&=\sup_{R\in\cB_{N}}
|R|^{\frac{1}{p_{+}(R)}-\frac{1}{p_{-}(R)}},
\\
C_2&=(\log N)^{2/p_{-}}.
\end{align*}

\noindent{\bf The estimate for $f_1$.} 
We shall verify that, 
if $\lm_1=C_1^{p_{-}}C_2$, then 
\begin{equation}\label{3.1}
\rho_{p(\cdot)}\l(\frac{T_kf_1}{\lm_1}\r)
=
\int_{\R^2}\l(\frac{T_kf_1(x)}{\lm_1}\r)^{p(x)}\,dx
\le C.
\end{equation}
It follows from H\"{o}lder's inequality that 
\begin{align*}
\lefteqn{
\rho_{p(\cdot)}\l(\frac{T_kf_1}{\lm_1}\r)
}\\ &=
\sum_{Q\in\cD_k}
\int_{Q}
\l(\frac{1}{\lm_1}\r)^{p(x)}
\l(\fint_{R(Q)}f_1(y)\,dy\r)^{p(x)}
\,dx
\\ &\le
\sum_{Q\in\cD_k}
\int_{Q}
\l(\frac{1}{\lm_1}\r)^{p(x)}
\l(\fint_{R(Q)}f_1(y)^{\frac{p_{-}(R(Q))}{p_{-}}}\,dy\r)^{\frac{p_{-}p(x)}{p_{-}(R(Q))}}
\,dx
\\ &=
\sum_{Q\in\cD_k}
\int_{Q}
\l(\frac{1}{\lm_1}\r)^{p(x)}
\l(\frac{1}{|R(Q)|}\r)^{\frac{p_{-}p(x)}{p_{-}(R(Q))}}
\l(\int_{R(Q)}f_1(y)^{\frac{p_{-}(R(Q))}{p_{-}}}\,dy\r)^{\frac{p_{-}p(x)}{p_{-}(R(Q))}}
\,dx.
\end{align*}
There holds, for $|R(Q)|\ge 1$, 
$$
\l(\frac{1}{C_1^{p_{-}}}\r)^{p(x)}
\l(\frac{1}{|R(Q)|}\r)^{\frac{p_{-}p(x)}{p_{-}(R(Q))}}
\le
\l(\frac{1}{|R(Q)|}\r)^{p_{-}},
$$
where we have used $C_1\ge 1$ and 
$\frac{p(x)}{p_{-}(R(Q))}\ge 1$.
Also, there holds, for $|R(Q)|<1$, 
\begin{align*}
\lefteqn{
\l(\frac{1}{C_1^{p_{-}}}\r)^{p(x)}
\l(\frac{1}{|R(Q)|}\r)^{\frac{p_{-}p(x)}{p_{-}(R(Q))}}
}\\ &\le
\l(\frac{1}{|R(Q)|}\r)^{\frac{p_{-}p(x)}{p_{+}(R(Q))}-\frac{p_{-}p(x)}{p_{-}(R(Q))}}
\l(\frac{1}{|R(Q)|}\r)^{\frac{p_{-}p(x)}{p_{-}(R(Q))}}
\\ &=
\l(\frac{1}{|R(Q)|}\r)^{\frac{p_{-}p(x)}{p_{+}(R(Q))}}
=
\l(\frac{1}{|R(Q)|}\r)^{\frac{p_{-}p(x)}{p_{+}(R(Q))}-p_{-}}
\l(\frac{1}{|R(Q)|}\r)^{p_{-}}
\\ &\le
\l(\frac{1}{|R(Q)|}\r)^{p_{-}},
\end{align*}
where we have used 
$$
C_1
\ge
|R(Q)|^{\frac{1}{p_{+}(R(Q))}-\frac{1}{p_{-}(R(Q))}}
\text{ and }
\frac{p(x)}{p_{+}(R(Q))}\le 1.
$$
We see that by the definition of $f_1$ 
\begin{align*}
\lefteqn{
\l(\int_{R(Q)}f_1(y)^{\frac{p_{-}(R(Q))}{p_{-}}}\,dy\r)^{\frac{p_{-}p(x)}{p_{-}(R(Q))}}
}\\ &\le
\l(\int_{R(Q)}f_1(y)^{p(y)}\,dy\r)
^{\frac{p_{-}p(x)}{p_{-}(R(Q))}-p_{-}}
\l(\int_{R(Q)}f_1(y)^{\frac{p(y)}{p_{-}}}\,dy\r)^{p_{-}}
\\ &\le
\l(\int_{R(Q)}f_1(y)^{\frac{p(y)}{p_{-}}}\,dy\r)^{p_{-}},
\end{align*}
where we have used 
$$
\l(\int_{R(Q)}f_1(y)^{p(y)}\,dy\r)
^{\frac{p_{-}p(x)}{p_{-}(R(Q))}-p_{-}}
\le
\l(\int_{\R^2}f(y)^{p(y)}\,dy\r)
^{\frac{p_{-}p(x)}{p_{-}(R(Q))}-p_{-}}
\le 1.
$$
These yield 
$$
\rho_{p(\cdot)}\l(\frac{T_kf_1}{\lm_1}\r)
\le
\sum_{Q\in\cD_k}
\int_{Q}
\l(\frac{1}{C_2}\r)^{p(x)}
\l(\fint_{R(Q)}f_1(y)^{\frac{p(y)}{p_{-}}}\,dy\r)^{p_{-}}
\,dx.
$$
Therefore, 
since $R(Q)\supset Q$ and 
$\frac{p(x)}{p_{-}}\ge 1$, 
\begin{align*}
\rho_{p(\cdot)}\l(\frac{T_kf_1}{\lm_1}\r)
&\le
\frac{1}{(\log N)^2}
\int_{\R^2}K_{N}[f_1^{p(\cdot)/p_{-}}](x)^{p_{-}}\,dx
\\ &\le C
\int_{\R^2}f_1(x)^{p(x)}\,dx
\le C,
\end{align*}
where we have used \eqref{1.1}.

\noindent{\bf The estimate for $f_2$.} 
We shall verify that, 
if $\lm_2=C_2$, then 
\begin{equation}\label{3.2}
\rho_{p(\cdot)}\l(\frac{T_kf_2}{\lm_2}\r)
=
\int_{\R^2}\l(\frac{T_kf_2(x)}{\lm_2}\r)^{p(x)}\,dx
\le C.
\end{equation}
Since $f_2\le 1$,
we immediately see that
$$
F
=
\frac{1}{\lm_2}\fint_{R(Q)}f_2(y)\,dy
\le 1.
$$
Therefore, by Lemma \ref{lem3.2}, 
with $P(x)=(e+|x|)^{-2}$, 
\begin{align*}
\rho_{p(\cdot)}\l(\frac{T_kf_2}{\lm_2}\r)
&=
\sum_{Q\in\cD_k}
\int_{Q}
\l(\frac{1}{\lm_2}\fint_{R(Q)}f_2(y)\,dy\r)^{p(x)}
\,dx
\\ &\le
C\sum_{Q\in\cD_k}
\int_{Q}
\l(\frac{1}{\lm_2}\fint_{R(Q)}f_2(y)\,dy\r)^{p(\8)}
\,dx
+
C\sum_{Q\in\cD_k}
\int_{Q}P(x)^{p(\8)}\,dx.
\end{align*}
Since $p(\8)\ge 2$ and 
the cubes $Q\in\cD_k$ are disjoint, 
we can immediately estimate the second term:
$$
\sum_{Q\in\cD_k}
\int_{Q}P(x)^{p(\8)}\,dx
=
\int_{\R^2}P(x)^{p(\8)}\,dx
\le C.
$$
We shall estimate the first term. 
It follows that 
\begin{align*}
\lefteqn{
\sum_{Q\in\cD_k}
\int_{Q}
\l(\frac{1}{\lm_2}\fint_{R(Q)}f_2(y)\,dy\r)^{p(\8)}
\,dx
}\\ &\le
\frac{1}{(\log N)^2}
\sum_{Q\in\cD_k}
\int_{Q}K_{N}f_2(x)^{p(\8)}\,dx
\\ &\le C
\int_{\R^2}f_2(x)^{p(\8)}\,dx,
\end{align*}
where we have used \eqref{1.1}.
Since $f_2\le 1$ 
we can apply Lemma \ref{lem3.2} again,
\begin{align*}
\int_{\R^2}f_2(x)^{p(\8)}\,dx
&\le
C\int_{\R^2}f_2(x)^{p(x)}\,dx
+
C\int_{\R^2}P(x)^{p(\8)}\,dx
\le C.
\end{align*}
Altogether, we obtain \eqref{3.2}. 

\noindent{\bf Conclusion.} 
The estimates \eqref{3.1}, \eqref{3.2} and 
Lemma\ref{lem3.1} yield the theorem.

\end{document}